\documentclass{article}
\usepackage{graphicx}
\usepackage{amsthm}
\usepackage{latexsym}
\usepackage{multirow}
\usepackage{amsmath}
\usepackage{epsfig}
\usepackage{color}
\usepackage{amsfonts}
\usepackage{algorithm}

\newtheorem{thm}{Theorem}[section]

\newtheorem{lem}[thm]{Lemma}
\newtheorem{algth}[thm]{Algorithm}
\newtheorem{prop}[thm]{Proposition}

\newtheorem{defn}[thm]{Definition}
\newtheorem{rem}[thm]{Remark}

\newcommand{\R}{\mathbb{R}}
\newcommand{\N}{\mathbb{N}}

\begin{document}

\title{A subgradient method for  equilibrium problems  involving quasiconvex bifunctions\footnote{This paper is supported by the NAFOSTED.}}                      

\author{ Le Hai Yen\footnote{Institute of Mathematics, VAST; Email: lhyen@math.ac.vn}  and Le Dung Muu\footnote{Thang Long University and Institute of Mathematics, VAST;  Email: ldmuu@math.ac.vn}}


\maketitle
\begin{abstract}
In this paper we propose a subgradient algorithm for solving  the equilibrium problem where the bifunction may be quasiconvex with respect to the second variable. The convergence of the algorithm
is investigated. A numerical example for a generalized variational inequality problem is provided to demonstrate the behavior of the algorithm.
\end{abstract}


\section{Introduction}
In this paper we consider the equilibrium problem defined by the  Nikaido-Isoda-Fan inequality that is given as
$$
\textnormal{Find } x^*\in C \textnormal{ such that } f(x^*,y) \geq 0 \quad \forall y\in C, \eqno (EP)
$$
where $C$ is a closed convex set in $\R^n$, $f:\R^n \times \R^n \rightarrow \R\cup \{+\infty\}$ is a bifunction such that $f(x,x)=0$ for every $ (x,x) \in C\times C \subset \textnormal{dom } f$.
The interest of this problem is that it unifies many important problems such as the Kakutani fixed point, variational inequality, optimization, Nash equilibria and some other problems  \cite{Bi2018, Bl1994, Ko2001, Mu1992} in a convenient way. The inequality in (EP)  first was used in \cite{Ni1955} for a convex game model. The first result for solution existence of (EP) has been obtained by Ky Fan in \cite{Fa1972}, where the bifunction $f$ can be quasiconvex with respect to the second argument. After the appearance of the paper by Blum-Oettli \cite{Bl1994}, the problem (EP) is attracted much  attention of many authors, and some solution approaches have been developed for solving Problem (EP) see e.g.  the  monographs \cite{Bi2018, Ko2001}. A basic method  for solving  some classes of Problem (EP) is subgradient (or auxiliary subproblem) one \cite{MQ2009, QMH2008}, where at each iteration $k$, having the iterate $x^k$, the next iterate $x^{k+1}$ is obtained by solving the subproblem
$\min \{ f(x^k,y) + \frac{1}{2\rho_k} \|y-x^k\|^2: y \in C\}$, with $\rho_k >0$. This subproblem is a strongly convex program whenever $f(x^k,.)$ is convex. Since in the case of variational inequality, where $f(x,y) := \langle F(x),y-x\rangle$, the iterate $x^{k+1}= P_C(x^k-\frac{1}{2\rho_k}F(x^k))$, the gradient method  can be considered as an extension of   the projection one commonly used in non smooth convex optimization \cite{So1985} as well as in  the variational inequality \cite{Fa2003}. The projection algorithms for paramonotone equilibrium problems with $f$ being convex with respect to second variable  can be found in  \cite{AM2014, Sa2011,  Th2017}.     Note that when $f(x,.)$ is quasiconvex the subgradient method fails to apply because of the fact that the objective function    $f(x^k,y) + \frac{1}{2\rho_k} \|y-x^k\|^2$, in general,  is neither convex nor quasiconvex. To our best knowledge there is no algorithm for  equilibrium problems  where the bifunction $f$ is quasiconvex with respect to the second variable.

In this paper we propose a projection algorithm for solving Problem (EP), where the bifunction may be quasiconvex with respect to the second variable. In order to deal with quasiconvexity we employ the subdifferential for quasiconvex function first introduced in \cite{GP1973}, see also  \cite{Pe1998, Pe2000} for its properties and calculus rules. The subdifferential of a quasiconvex function has been used by some authors for nonsmooth quasiconvex optimization see. e.g. \cite{Ko2003, Ki2001, So1985}, and for quasiconvex feasibility problems \cite{Ce2005,Ni2016}.

             \section{Subdifferentials of quasi-convex functions}
First of all, let us recall the well known definition of quasiconvex functions, see e.g. \cite{Ma1969}

\begin{defn}
A function $\varphi:\R^n \rightarrow \R \cup \{+\infty\}$ is called quasiconvex on a convex subset $C$ of $ \R^n$ if and only if for every $x,y \in C$ and $\lambda \in [0,1]$, one has
\begin{equation}
\varphi[(1-\lambda)x +\lambda y] \leq \max[\varphi(x), \varphi(y)].\label{eq1}
\end{equation}
\end{defn}
It is easy to see that $\varphi$ is quasiconvex on $C$   if and only if the level set
\begin{equation}
L_{\varphi}(\alpha):=\{ x\in C: \quad \varphi(x)<\alpha\}.\label{eq2}
\end{equation}
is convex for any $\alpha\in \R$.

We consider the following  Greenberg-Pierskalla subdifferential  introduced in \cite{GP1973}
\begin{eqnarray}
\partial^{GP} \varphi(x):=\{ g\in \R^n: \langle g, y-x\rangle \geq 0 \quad \Rightarrow \varphi(y) \geq \varphi (x)\}.
\end{eqnarray}

A variation of the GP-subdifferential is the star-sudifferential that is defined as
\begin{equation}
\partial^*\varphi(x):=
\begin{cases}\{ g\in \R^n: \langle g, y-x\rangle > 0 \Rightarrow \varphi(y) \geq \varphi(x)\}& \textnormal{ if }\  x\not\in D_*\\
\R^n & \textnormal{ if } \ x\in D_*, \nonumber \\
\end{cases}
\end{equation}
where $D_*$ is the set of minimizers of $\varphi$ on $\R^n$.
If $\varphi$ is continuous on $\R^n$, then $\partial^{GP}\varphi(x) = \partial^*\varphi (x)$ (\cite{Pe2000}).

These subdifferentials are also called quasi-subdifferentials or the normal-subdifferentials.  Some calculus rules and optimality conditions   for these subdifferentials have been studied in \cite{Pe1998,Pe2000}, among them the following results will be used in the next section.


\begin{lem}(\cite{Ki2001}, \cite{Pe1998})
\label{lem1} Assume that $\varphi:\R^n \rightarrow \R \cup \{+\infty\}$ is upper semicontinuous and quasiconvex on $dom \varphi$. Then
\begin{equation}
\partial^{GP} \varphi(x) \not= \emptyset \quad \forall x\in dom \varphi,
\end{equation}
and
\begin{equation}
\partial^{*} \varphi(x)= \partial^{GP} \varphi(x) \cup \{0\}.
\end{equation}

\end{lem}

\begin{lem}(\cite{Ki2001}, \cite{Pe2000})
\label{lem2} \begin{equation}
0\in \partial^{GP} \varphi(x) \Leftrightarrow \partial^{GP} \varphi(x)=\R^n \Leftrightarrow x \in \textnormal{argmin }_{y \in \R^n} \varphi (y).
\end{equation}
\end{lem}

The following result follows from Lemma 4 in \cite{Ki2001}, which can be used  to find the subdifferential of a fractional quasiconvex function.
\begin{lem}(\cite{Ki2001})
\label{lem3}
Suppose $\varphi(x)=a(x)/b(x)$ for all $x\in dom \varphi$, where $a$ is a convex function, $b$ is finite and positive on $dom \varphi$, $dom \varphi$ is convex and one of the following conditions holds
\begin{enumerate}
\item[(a)] $b$ is affine;
\item[(b)] $a$ is nonnegative on $dom \varphi$ and $b$ is concave;
\end{enumerate}
Then $\varphi$ is quasiconvex and $\partial (a-\alpha b) (x)$ is a subset of $\partial^{GP}\varphi(x)$ for  $\alpha=\varphi(x)$, where $\partial$ stands for the subdifferential in the sense of convex analysis.
\end{lem}

\section{Algorithm and its convergence analysis}

In this section we propose an algorithm for solving Problem (EP), by using the star-sudifferential with respect to the second variable of  the bifunction $f$.
 As usual, we suppose that  $C$ is a nonempty closed convex subset in $\R^n$, $f:\R^n \times \R^n \rightarrow R\cup\{+\infty\}$ is a bifunction  satisfying $f(x,x)=0$ and  $C\times C \subset \textnormal{dom } f$. The solution set of Problem (EP) is denoted by $S(EP)$.

\vskip1cm
\textbf{Assumptions}
\begin{enumerate}
\item[(A1)] For every $x\in C$, the function $f(x,.)$ is quasiconvex and $f(.,.)$ is upper semi continuous on an open set containing $C\times C$;
\item[(A2)] The bifunction $f$ is pseudomonotone on $C$, that is
\begin{equation}
f(x,y)\geq 0 \Rightarrow f(y,x)\leq 0 \quad \forall x\in C, y\in C,\nonumber
\end{equation}
and paramonotone on $C$ with respect to $S(EP)$, that is
\begin{equation}
x\in S(EP), y\in C \textnormal{ and } f(x,y)=f(y,x)=0 \Rightarrow y\in S(EP).\nonumber
\end{equation}
\item[(A3)] The solution set $S(EP)$ is nonempty.
\end{enumerate}

Paramonotonicity of  bifunctions has been  used for equilibrium as well as split equilibrium problems in some papers see e.g. \cite{AM2014, Sa2011,Th2017,Ye2019}. Various properties of paramonotonicity can be found, for example   in \cite{Ju1998}.

 For simplicity of notation. let us denote the sublevel set of the function $f(x,.)$ with value $0$ and the star subdifferential of $f(x,.)$ at $x$ as follows
\begin{eqnarray}
L_f(x)&:=& \{ y: f(x,y)< f(x,x)=0\},\nonumber\\
\partial^*_2 f(x,x)&:=&\{g\in \R^n |\langle g,y-x \rangle <0  \quad\forall y\in L_f(x)\}.\nonumber
\end{eqnarray}
The projection algorithm below can be considered as an extension of the one in \cite{Sa2011} to equilibrium problem (EP), where $f$ is quasiconvex with respect to the second variable.
\vskip0.5cm
\begin{algth}(The Normal-subgradient method)
\label{algth1}
 Take a real sequence $\{\alpha_k\}$ satisfying  the following conditions
\begin{eqnarray}
 &\quad \alpha_k >0 \quad \forall k\in \N,&\nonumber\\
&\sum_{k=1}^{\infty} \alpha_k=+\infty, \quad \sum_{k=1}^{\infty} \alpha_k^2 <+\infty.&\nonumber
\end{eqnarray}
\\\textbf{Step 0:} $x^0\in C$, $k=0$.
\\\textbf{Step k:} $x^k \in C$.

Take $g^k\in \partial_2^* f(x^k,x^k)$.

If $g^k = 0$,  \textbf{stop}: $x^k$ is a solution.

Else normalize $g^k$ to obtain   $\|g^k\|=1.$

Compute
\begin{equation}
x^{k+1} =P_C(x^k -\alpha_k g^k).\nonumber
\end{equation}

If $x^{k+1}=x^k$ then \textbf{stop}. 

Else update $k\longleftarrow k+1$.

\end{algth}
\begin{rem}
\label{rem0}
At each iteration $k$, in order to  check that  the iterate $x^k$ is a solution or not one can solve   the programming problem
  $\min_{y\in C} f(x^k,y)$.  In general solving this problem  is costly, however  in some special cases, for example,  when $C$ is a polyhedral convex set and the function $f(x^k,.)$ is affine fractional, this program can be solved efficiently by linear programming methods.
\end{rem}
The following remark ensures the validity of the algorithm.
\begin{rem}
\label{rem1a}

(i)   Since the star-subdifferential is a cone, one can always normalize its nonzero element to obtain an unit vector in the subdifferential.

(ii) If Algorithm \ref{algth1} generates a finite sequence, the last point is a solution of (EP).
Indeed,  if $0\in \partial_2^* f(x^k,x^k)$  then by Lemma \ref{lem2}, we have
\begin{equation}
x^k \in \textnormal{argmin}_{y\in \R^n} f(x^k,y),\nonumber
\end{equation}
which implies  $f(x^k,y)\geq 0$ for any $y\in C$.

If $x^{k+1}=x^k$ then $x^k= P_C(x^k-\alpha_kg^k)$. Hence,
$$\langle -\alpha_kg^k, y-x^k\rangle \leq 0 \quad \forall y\in C,$$
or
$$\langle g^k, y-x^k\rangle \geq 0 \quad \forall y\in C.$$
 Since $g^k \in \partial_2^* f(x^k,x^k)= \partial_2^{GP} f(x^k,x^k)$, we have   $f(x^k,y) \geq  f(x^k,x^k) \geq 0$ for every $y\in C$.
\end{rem}




For convergence of the algorithm we need the following results that has been proved in \cite{Sa2011} for the case the bifunction $f$ is convex in the second variable.
\begin{lem}
\label{lem4}
The following inequality holds true for every $k$,
\begin{equation}
\|x^{k+1}-x^{k}\| \leq \alpha_k.
\end{equation}
\end{lem}
\begin{proof}
Since $x^{k+1}=P_C(x^k -\alpha_k g^k)$,
\begin{equation}
\langle x^{k+1} -x^k+\alpha_kg^k, y -x^{k+1} \rangle\geq 0 \quad \forall y \in C.\nonumber
\end{equation}
By subtituting $y=x^k$, we obtain
\begin{eqnarray}
\|x^{k+1} -x^k\|^2 &\leq & \alpha_k\langle g^k, x^k -x^{k+1} \rangle\nonumber\\
&\leq& \alpha_k \|g^k\|\|x^{k+1}-x^k\| \nonumber\\
&=& \alpha_k \|x^{k+1}-x^k\|.\nonumber
\end{eqnarray}

Therefore, $$\|x^{k+1}-x^k\| \leq \alpha_k.$$
\end{proof}

\begin{prop}
\label{prop1}
For every  $z\in C$ and $k$, the following inequality holds
\begin{equation}
\|x^{k+1}-z\|^2 \leq \|x^k -z\|^2 +2\alpha_k \langle g^k, z-x^k\rangle +\alpha_k^2.\label{eq11}
\end{equation}

\end{prop}
\begin{proof}
Let $z\in C$, then we have
\begin{eqnarray}
&&\|x^{k+1}-z\|^2\nonumber\\&=&\|x^k-z\|^2 -\|x^{k+1}-x^k\|^2 +2\langle x^k-x^{k+1}, z-x^{k+1}\rangle\nonumber\\
&\leq& \|x^k-z\|^2 +2\langle x^k-x^{k+1}, z-x^{k+1}\rangle.\label{eq12}
\end{eqnarray}
Since $x^{k+1}=P_C(x^k -\alpha_k g^k)$ and $z\in C$,
\begin{eqnarray}
\langle x^{k} -x^{k+1}, z-x^{k+1}\rangle \leq \alpha_k \langle g^k, z-x^{k+1}\rangle.\label{eq13}
\end{eqnarray}
From (\ref{eq12}) and (\ref{eq13}),
\begin{eqnarray}
&&\|x^{k+1}-z\|^2 \nonumber\\&\leq & \|x^k-z\|^2 +2\alpha_k \langle g^k, z-x^{k+1}\rangle\nonumber\\
&=&\|x^k-z\|^2 +2\alpha_k \langle g^k, z-x^k\rangle\nonumber\\
&&+ 2\alpha_k \langle g^k, x^k-x^{k+1}\rangle.\label{eq14}
\end{eqnarray}
By Cauchy-Schwart inequality and the fact that $\|g^k\|=1$, we have
\begin{equation}
\langle g^k, x^k-x^{k+1}\rangle \leq \|x^k-x^{k+1}\|.
\nonumber\end{equation}
Since, by Lemma \ref{lem4},  $\|x^k-x^{k+1}\|\leq \alpha_k$, the inequality    (\ref{eq14}) becomes
\begin{equation}
\|x^{k+1}-z\|^2 \leq \|x^k -z\|^2 +2\alpha_k \langle g^k, z-x^k\rangle +2\alpha_k^2\nonumber
\end{equation}

\end{proof}

The following lemma is  an extension of Lemma 6  in \cite{Ki2001} to the diagonal subdifferential of a bifunction being quasiconvex in its second variable.
\begin{lem}
\label{lem5}
\begin{itemize}
\item[(a)] If $B(\overline{x}, \epsilon) \subset L_f(x^k)$ for some $\overline{x}\in \R^n$ and $\epsilon \geq 0$, then
\begin{equation}
\langle g^k, x^k-\overline{x}\rangle >\epsilon.\nonumber
\end{equation}
\item[(b)]$\liminf_{k \rightarrow \infty} \langle g^k, x^k-\overline{x} \rangle \leq 0 \quad \forall \overline{x} \in C $.
\end{itemize}
\end{lem}
\begin{proof}
\begin{itemize}
\item[(a)] If $B(\overline{x}, \epsilon) \subset L_f(x^k)$, then $\overline{x}+\epsilon g^k \in L_f(x^k)$ for $\epsilon >0$ small enough. Since $g^k\in \partial_2^* f(x^k,x^k)$, we have
\begin{equation}
\langle g^k, \overline{x}+\epsilon g^k -x^k\rangle <0.\nonumber
\end{equation}
From $\|g^k\|=1$, it follows  that
\begin{equation}
\langle g^k, x^k-\overline{x}\rangle >\epsilon.\nonumber
\end{equation}
\item[(b)] By contradiction, we assume that there exist $\overline{x}\in C$, $\xi >0$ and $k_0$ such that for any $k\geq k_0$,
\begin{equation}
\langle g^k, x^k-\overline{x}\rangle \geq \xi > 0.\nonumber
\end{equation}
In view of  Proposition \ref{prop1}, we have
\begin{equation}
2\alpha_k \langle g^k, z-x^k\rangle \leq \|x^k -z\|^2 -\|x^{k+1}-z\|^2 +\delta_k.\nonumber
\end{equation}
By summing up, we obtain
\begin{eqnarray}
2\sum_{k=0}^\infty \alpha_k \epsilon &\leq & 2\sum_{k=0}^\infty \alpha_k\langle g^k, z-x^k\rangle\nonumber\\
&\leq &\|x^0-z\|^2 +\sum_{k=1}^{\infty} \alpha_k^2.\nonumber
\end{eqnarray}
Thus,
$$0<\epsilon \leq \frac{\|x^0-z\|^2 +\sum_{k=1}^{\infty} \alpha_k^2}{2\sum_{k=0}^\infty \alpha_k}.$$
This is a contradiction because of $\sum_{k=1}^{\infty} \alpha_k^2<+\infty$ and $\sum_{k=0}^\infty \alpha_k=+\infty$.
\end{itemize}
\end{proof}

Now, we can state the convergence  theorem.
\begin{thm}
Suppose that Algorithm \ref{algth1} does not terminate. Let $\{x^k\}$ be the infinite sequence generated by the  algorithm  and   $\{x^{k_q}\}$ be the subsequence of
$ \{x^k\}$ that  contains all iterates belonging to $S(EP)$.
Then  under Assumptions (A1), (A2), (A3), one has
\begin{itemize}
\item[(a)] If $\{x^{k_q}\}$ is infinite then $\lim_{k \rightarrow \infty} d(x^k, S(EP))=0;$
\item[(b)] If $\{x^{k_q}\}$ is finite then the sequence $\{x^k\}$ converges to a solution of the problem (EP).
\end{itemize}

\end{thm}

\begin{proof} Let $x^*$ be an arbitrary point of $S(EP)$.
 If $ f(x^*,x^k) \geq 0,$ then
by pseudomonotonicity,
$f(x^k, x^*) \leq 0.$
If $f(x^k,x^*)=0$, then again by pseudomonotonicity, we obtain $f(x^*,x^k) = 0.$  Moreover, by paramonotonicity of $f$ on $C$ w.r.t $S(EP)$, we can say that $x^k$ is also a solution of (EP).

We consider two cases.
\begin{itemize}
\item[(a)] \textbf{The sequence $\{x^{k_q}\}$ is infinite.} By Proposition \ref{prop1},
\begin{equation}
\|x^{k+1}-z\|^2 \leq \|x^k -z\|^2 +2\alpha_k \langle g^k, z-x^k\rangle +\alpha_k^2\nonumber
\end{equation}
By choosing $z=x^k_q$, we obtain
\begin{equation}
\|x^{k+1}-x^{k_q}\|^2 \leq \|x^k -x^{k_q}\|^2 +2\alpha_k \langle g^k, x^{k_q}-x^k\rangle +\alpha_k^2\label{eq14a}
\end{equation}
Now, for $k$ such that $k_q<k<k_{q+1}$, $x^k$ does not belong to $S(EP)$. Then, for $k_q<k<k_{q+1}$, $f(x^k,x^{k_q})<0$ , which means that $x^{k_q}\in L_f(x^k)$. Therefore,
\begin{equation}
\langle g^k, x^{k_q}-x^k \rangle <0 \quad \forall k_q<k<k_{q+1}.\label{eq15a}
\end{equation}
From (\ref{eq14a}) and (\ref{eq15a}), for $k_q\leq k<k_{q+1}$, it follows that
\begin{equation}\label{eqM16}
\|x^k-x^{k_q}\|^2 \leq \sum_{i=k_q}^k \alpha_i^2\leq \sum_{i=k_q}^{\infty} \alpha_i^2.
\end{equation}
Moreover, since $x^{k_q}\in S(EP)$, we have $d^2(x^k, S(EP))\leq \|x^k-x^{k_q}\|^2$. In addition, by the assumption \\$\lim_{q\rightarrow \infty} \sum_{i=k_q}^{\infty} \alpha_i^2 < +\infty$, it follows from (\ref{eqM16}) that,
\begin{equation}
\lim_{k \rightarrow \infty} d(x^k, S(EP)=0.
\end{equation}
\vskip1cm
\item[(b)]\textbf{The sequence $\{x^{k_q}\}$ is finite}. Then, there exists $k_0$ such that for $k\geq k_0$, we have $x^k \not\in S(EP)$. Then
$f(x^k,x^*)<0.$
In this case, we divide our proof into four steps.

\textbf{Step 1:}$\{\|x^k-x^*\|\}$ is convergent and therefore $\{x^k\}$ is bounded.
\\Indeed, $f(x^k,x^*)<0$ for $k\geq k_0$, which means that  $x^*\in L_f(x^k)$ for $k\geq k_0$. Thus,
\begin{equation}
\langle g^k, x^*-x^k \rangle <0 \quad \forall k\geq k_0.\label{eq15}
\end{equation}
Combining this with (\ref{eq11}) in Proposition \ref{prop1}, we obtain
\begin{equation}
\|x^{k+1}-x^*\|^2 \leq \|x^k -x^*\|^2  +\alpha_k^2.\label{eq16}
\end{equation}
Since $\sum_{k=1}^{\infty} \alpha_k^2 <+\infty$, we can conclude that $\{\|x^k-x^*\|\}$ is convergent and therefore $\{x^k\}$ is bounded.

\textbf{Step 2:} $\liminf_{k \rightarrow \infty} \langle g^k, x^k-x^* \rangle = 0$.
Indeed, thanks to Lemma \ref{lem5}(b), we have
\begin{equation}
\liminf_{k \rightarrow \infty} \langle g^k, x^k-x^* \rangle \leq 0.\label{eq17}
\end{equation}
From (\ref{eq15}) and (\ref{eq17}),
\begin{equation}
\liminf_{k \rightarrow \infty} \langle g^k, x^k-x^* \rangle = 0.\nonumber
\end{equation}

\textbf{Step 3:} Let $\{x^{k_i}\}$ be a subsequence of $\{x^k\}$ such that
\begin{equation}
\lim_{q\rightarrow \infty} \langle g^{k_i}, x^{k_q}-x^* \rangle=\liminf_{k \rightarrow \infty} \langle g^k, x^k-x^* \rangle = 0.\label{eq18}
\end{equation}
Since $\{x^k\}$ is bounded, then $\{x^{k_i}\}$ is bounded too. Let $\overline{x}$ be a limit point of $\{x^{k_i}\}$, without loss of generality, we assume that $\lim_{i\rightarrow \infty} x^{k_i}=\overline{x}$.
\\We now prove that $f(\overline{x}, x^*)=0$.

 Indeed, by pseudomonotonicity of $f$ on $C$, we have
$$f(\overline{x}, x^*)\leq 0.$$
Moreover, we show  that $f(\overline{x}, x^*) = 0.$ In fact, by contradiction,  assume that there exists $a>0$ such that
$$f(\overline{x}, x^*)\leq -a.$$
Then, since $f(., .)$ is upper semicontinuous on $C\times C$,  there exist $\epsilon_1 >0, \epsilon_2 > 0 $ such that  for any $x\in B(\overline{x}, \epsilon_1)$, $y\in B(x^*, \epsilon_2)$ we have
$$f(x, y)\leq -\frac{a}{2}.$$
On the other hand, since $\lim_{i\rightarrow \infty} x^{k_i}=\overline{x}$,  there exists $i_0$ such that  $x^{k_i}$ belongs to $B(\overline{x}, \epsilon_1)$ for every $i\geq i_0$,
So, for $i\geq i_0$ and $y\in B(x^*, \epsilon_2)$,  one has
$$f(x^{k_q}, y)\leq -\frac{a}{2} < 0,$$
which means that $B(x^{*}, \epsilon_2)\subset  L_f(x^{k_q})$. Then, by Lemma \ref{lem5}(a), we have
\begin{equation}
\langle g^{k_q}, x^{k_q}-x^{*}\rangle >\epsilon_2 \ \forall q\geq q_0.\nonumber
\end{equation}
This contracts with (\ref{eq18}). So, $f(\overline{x}, x^*)=0$.

\textbf{Step 4:}   The sequence $\{x^k\} $ converges to a solution of (EP). Indeed, by Step 3,  $f(\overline{x}, x^*)=0$. Since $f$ is pseudomonotone on $C$, we have $f(x^*, \overline{x}) \leq 0$, which together with $x^* \in S(EP)$,  implies  $f(x^*,\overline{x} )=0$.  Then, thanks to the paramonotonicity of $f$,  $\overline{x}$ is also a solution of (EP). By Step 1,  $\{\|x^k-\overline{x}\|^2\}$ is convergent, which together with $\lim_{q\rightarrow \infty} x^{k_q}=\overline{x}$ implies that the whole sequence  $\{x^k\} $ must  converge to the solution $\overline{x}$  of (EP).
\end{itemize}
\end{proof}

\begin{rem} (i) If at each iteration $k$, one can check  that whether $x^k$ is a solution or not yet, then all generated iterates do not belong to the solution set $S(EP)$. In this case  the  sequence $\{x^k\}$ converges to a solution of (EP).

(i) If, in addition,  we  assume that
 $f(x,.)$ is strictly quasiconvex for $x\in C$, that is for every $\langle g,y-x\rangle > 0 \
\Rightarrow f(x,y) > f(x,x) = 0$ whenever $g \in \partial f^*(x,x)$,
   then  the sequence $\{x^k\}$ converges to the unique solution of (EP).

Indeed, if $f(x,.)$ is strictly quasiconvex,  then, since $f(x^k,x^*) \leq 0 = f(x^k,x^k)$, by strictly quasiconvex,   $\langle g^k, x^*-x^k\rangle \leq 0$. thus in virtue of Proposition \ref{prop1}, we obtain
$$\|x^{k+1} -x^*\|^2 \leq \|x^k-x^*\|^2 + \alpha_k^2 \  \forall k.$$
 So the convergence of   $\{x^k\}$ can be proved as in case (b) of the convergence theorem.
 \end{rem}

 \section{A Generalized Variational Inequality and  Computational Experience}
 Consider the following  generalized variational inequality problem
$$
\textnormal{Find } x^*\in C \textnormal{ such that } \langle F(x), \varphi (y) - \varphi (x) \rangle  \geq 0 \quad \forall y\in C, \eqno (GVI)
$$
where $F : C\to \R^n$  is a given operator with dom$F \subseteq C$, and $\varphi : C\to \R^n $  is a vector function  such that  $\langle F(x), \varphi(y)\rangle$ is quasiconvex on $C$ for any fixed $x\in C $.  Generalized variational inequalities  were considered in \cite{Noor1987}, and some solution-existence results were established there. Problem (GVI) can take the form of equilibrium (EP) by taking  $f(x,y):=   \langle F(x), \varphi (y) - \varphi (x) \rangle$.
 Clearly, when $\varphi (y) \equiv y$, Problem (GVI)  becomes  a standard variational inequality one. Note that  because of quasiconvexity of $\varphi$, available methods for variational inequalities as well as those for equilibrium problems with $f$ being convex in its second variable fail to apply to (GVI). Now we consider a typical example of problem (GVI) by taking

\begin{equation}
f(x,y)=\left\langle Ax+b, \frac{A_1y+b_1}{c^Ty+d}-\frac{A_1x+b_1}{c^Tx+d}\right\rangle,
\end{equation}
with $A,A_1\in \R^{n\times n}$, $b,b_1,c\in \R^n$, $d\in \R$ and $C\subset \{x| \quad c^Tx+d >0\}$.

Let $\hat{A}= (dA_1^T-cb_1^T)A$.
Then, the bifunction $f(x,y)$ is paramonotone if and only if $\hat{A}_1=\frac{1}{2}(\hat{A}+\hat{A}^T)$ is positive semidefinite and $\textnormal{rank}(\hat{A}_1)\leq \textnormal{rank}(\hat{A})$ (see \cite{Ju1998}).

We  apply the  two versions of the Normal-subgradient algorithm in Section 3 to solve this problem. In the first version (denoted by (\textbf{NG1})), we construct the  sequence   without checking the solvability of each obtained iterate. In the second version (denoted by (\textbf{NG2})), we check at each iteration whether  $x^k$ is a solution of $(EP)$ or not  by solving the linear fractional program $\min_{y\in C} f(x^k,y)$. As it is well known, see e.g. \cite{Sch1981}, that affine fractional functions appear in many practical problems in various fields.

We take $C=\left[ 1,3 \right]^n$,  the  entries of $A,A_1,b,b_1,c,d$ are uniformly generated in the interval $\left[0,1\right]$. We tested   the algorithm with  problems of different sizes, each of sizes has hundred instances  We stop the computation of the first version (\textbf{NG1}) at iteration $k$ if either $ g^k=0$ or $err_1:=\|x^k-x^{k+1}\|<10^{-4}$, whereas   for the second version (\textbf{NG2}), at each iteration $k$, we check the solvability of $x^k$ by solving the problem $y^k=\textnormal{argmin}_{y\in C} f(x^k,y),$
and stop the computation at iteration $k$ if $-\min_{y\in C} f(x^k,y) <10^{-3}$. Otherwise,    we stop the computation if  the number of iteration exceeds $2000$. For both versions, we say that a problem is  successfully solved  if  an iterate $x^k$ satisfying \\$err:=-\min_{y\in C} f(x^k,y) <10^{-1}$ has been obtained.  Figure \ref{fig1} shows that the error  goes to 0 as the number of iterations
k goes to $+\infty$. The  computational results are shown in Table \ref{tab1} for (\textbf{NG1}) and Table \ref{tab2} for (\textbf{NG2}), where the  number of successfully solved problems as well as the  time and  error in average  are reported. 
\begin{center}
\begin{figure}
\centering 
\includegraphics[scale=0.35]{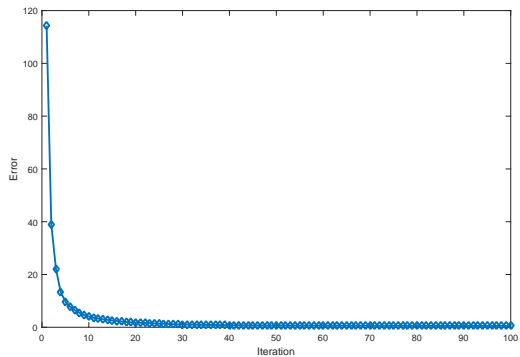}
\caption{Behavior of the error}
\label{fig1}
\end{figure}
\end{center}
\begin{table}
\caption{Algorithm (\textbf{NG1}) with $\alpha_k=\dfrac{100}{k+1}$} 
\centering 
\begin{tabular}{c c c c c c } 
\hline\hline 
 n &N. prob. & N. succ. prob.  & CPU-times(s)& Error\\ [0.5ex] 
\hline 
5&100& 100& 0.026574& 0.000006 \\
10&100&  100& 0.277551& 0.000308\\
20&100& 100& 0.582287& 0.001506   \\
50&100& 87& 9.787128& 0.027892  \\
\hline 
\end{tabular}
\label{tab1} 
\end{table}
\begin{table}
\caption{Algorithm (\textbf{NG2}) with $\alpha_k=\dfrac{100}{k+1}$} 
\centering 
\begin{tabular}{c c c c c c } 
\hline\hline 
 n &N. prob. & N. succ. prob.  & CPU-times(s)& Error\\ [0.5ex] 
\hline 
5&100& 100& 0.005953& 0.000004 \\
10&100&  100& 0.124109& 0.000066 \\
20&100& 100& 0.527599& 0.000625  \\
50&100& 100& 6.685981& 0.003728  \\
\hline 
\end{tabular}
\label{tab2} 
\end{table}


\end{document}